\documentclass[12pt]{article}

\usepackage{latexsym,amssymb,amsmath}
\newcommand{\Z}{\mathbb Z}
\newcommand{\N}{\mathbb N}
\newtheorem{lem}{Lemma}[section]
\newtheorem{defn}[lem]{Definition}

\newtheorem{co}[lem]{Corollary}
\newtheorem{thm}[lem]{Theorem}
\newtheorem{prop}[lem]{Proposition}

\newenvironment{proof}{\textbf{Proof.}}{\newline\hspace*{\fill}{$\Box$}\\}

\begin{document}
\title{Growth in infinite groups of infinite subsets}
\author{J.\,O.\,Button\\
Selwyn College\\
University of Cambridge\\
Cambridge CB3 9DQ\\
U.K.\\
\texttt{jb128@dpmms.cam.ac.uk}}
\date{}
\maketitle
\begin{abstract}
Given an infinite group $G$,
we consider the finitely additive measure 
defined on finite unions of cosets of finite index subgroups. We show
that this shares many properties with the size of subsets of a finite
group, for instance we can obtain equivalent results on Ruzsa distance
and product free sets. In particular if $G$ has infinitely many finite
index subgroups then it
has subsets $S$ of measure arbitrarily close to 1/2 with square $S^2$
having measure less than 1. 
\end{abstract}

\section{Introduction}

The theory of the growth of finite subsets $A,B$ in a group $G$ is well
established: one considers the size of the products sets $AB,A^2,B^2$
and higher powers in terms of $|A|$ and $|B|$. One hopes to be able to
say something about the structure of such sets if any or all of
$|AB|,|A^2|,|A^3|,\ldots$ are comparatively small. The subject
originated with abelian groups but recently results have been
established for non abelian groups too, as discussed in \cite{taop} and
this will be our focus here.

One often takes $G$ to be finite, although in this context we do not
really stick to one group but allow $G$ to vary over an infinite family
of groups (such as the non abelian finite simple groups) or even over
all finite groups. There are also results for infinite groups, where
we might consider one group $G$ only or again a whole class. We just
mention two recent results in this regard: the first, which is Theorem 2
in \cite{borpas}, covers torsion free groups $G$. It states that
there exists a polynomial $p(n)=32(n+3)^6$ such that if $A$ is a finite
subset of any torsion free group $G$ which is not contained in a left
coset of a cyclic subgroup and $|B|>p(n)$ then $|AB|>|A|+|B|+n$. The other
result is that of Safin in \cite{saf} which states that there is $c>0$ such
that $|A^3|\ge c|A^2|$ for $A$ any finite subset of the non abelian free
group $F_2$ of rank 2 with $\langle A\rangle$ not cyclic.

However if our group $G$ is infinite then no matter how quickly the
sequence $|A^n|$ grows, we have that the proportion of elements of
$G$ in $A^n$ is zero. We would like to assign size or measure to infinite
subsets $A,B$ of $G$ and their product sets, in such a way that the measure
of $G$ is finite in order to see if similar results apply. Of course there
are several ways to do this, such as if $G$ is a compact group equipped
with a Haar measure on the Borel subsets of $G$. However in this paper
our groups of interest will often be countable, so any countably additive
measure defined on all subsets of $G$
which is invariant under left multiplication will either
assign 0 to all $g$ in $G$, and hence to $G$, or infinite measure to $G$.

A successful way round this for some countable discrete groups is to
observe that asking for countable additivity is expecting too much and
so one requires a finitely additive left invariant measure $\mu$ with
$\mu(G)=1$ defined on all subsets of $G$, leading us to the important
class of amenable groups. But many countable groups are not amenable,
including any group containing the free group $F_2$ (the only specific
group mentioned so far).

Another way out of this, at least for infinite groups $G$ which are
finitely generated, is to work in the profinite completion
$\hat{G}$ which is a compact group with Haar measure, obtained
from the finite quotients of $G$. Our approach in this paper has
connections with this, in that it also looks at the finite quotients
of a group, but is much simpler. We have already mentioned that
countable additivity is a very strong property to expect in a measure.
Moreover so is having it defined on all subsets of a group, as we see
with the reals and Lebesgue measure. Thus we adopt the more relaxed
approach of merely requiring our measure to be finitely additive on
a class of
subsets which are closed under finite unions, finite intersections,
and complementation. However we require some sets to be measurable, and
any sensible notion of size should assign $1/i$ to a coset $gH$ of a
subgroup $H$ that has index $i$ in $G$. Therefore we leave it at that and
only take finite unions of cosets of finite index subgroups for our
measurable sets, forming the basic measure as defined in Section 2.
Some of our results also hold for any left invariant measure
extending the basic measure, which we term an acceptable measure.

In Section 3 we look at the standard concepts for product sets in finite
groups, especially in the non abelian case as covered in \cite{taop}
and featuring such topics as Ruzsa distance. We show that they easily 
generalise to any infinite group $G$ equipped with the basic measure.
A big plus here is that, just as with finite sets, if $A$ is a subset
of $B$ with the basic measure of $A$ equal to that of $B$ then $A=B$.
This allows us as in the finite case to identify measurable subsets $A$
satisfying certain conditions on their measure, for instance in
Corollary 3.9 where $A^n$ is a proper subset of $G$ for all integers
$n$. We also point out that even for finite groups we can have
generating sets $S$ which are symmetrised but with no $n$ such that
$S^n=G$.

In Section 4 we are interested in comparing the size $|A|$ of a measurable
subset $A$ with $|A^2|$. There has been much work done on this for finite
groups and we first review such results, especially the problem of
trying to find subsets $A$ with $A^2\neq G$ such that $|A|$ is as large
as possible. We obtain Corollary 4.8 which says that if $G$ is an infinite
finitely generated linear group in characteristic zero then there is
always $A\subseteq G$ with $|A|/|G|=1/2$ but such that $A^2\neq G$. This
is achieved by showing in Theorem 4.4 that the same statement is true
when $G$ is any finite group of even order. Another example of a concept
in finite groups that applies here is that of a product free subset $S$ of a
group $G$, which is one where $S$ and $S^2$ do not intersect, and a recent
result of Gowers shows that for any $\epsilon >0$ there is a finite simple
group $G_\epsilon$ such that $|S|/|G|<\epsilon$ for any product free subset
$S$ of $G_\epsilon$. We show that this is still true with $G_\epsilon$ 
an infinite finitely generated residually finite group by applying a
construction of Kassabov and Nikolov which uses profinite groups.

In the last Section, we explore the fact that for any group $G$ Ruzsa
distance is a genuine metric on the set ${\cal S}(G)$ of the finite index
subgroups of $G$. Proposition 5.2 tells us when the triangle inequality
becomes an equality, allowing us to turn this set into a metric graph.
Consequently we have a conjugation action by isometries of $G$ on this
graph and we show that this action is faithful if and only if
$G$ is residually finite with trivial centre. 

\section{Definitions and results}
\begin{defn} Given a set $\Omega$, an {\bf algebra} $\cal A$ of $\Omega$
is a collection of subsets of $\Omega$ with $\Omega\in\cal A$ such that
$A,B\in\cal A$ implies that $A\cup B$ and $A\backslash B\in\cal A$.
\end{defn}

This implies that $\emptyset\in\cal A$. Moreover if $A_1,\ldots , A_n\in\cal A$
then so is $\cup_{i=1}^n A_i$ and $\cap_{i=1}^n A_i$.

\begin{defn} Let $\Omega$ be a set and $\cal A$ an algebra of $\Omega$.
A (finitely additive)
{\bf measure} on $\cal A$ is a non-negative function from $\cal A$ to
$\mathbb R$, where we write $|A|$ for the image of $A$ in $\cal A$,
such that if $A,B\in\cal A$ with $A\cap B=\emptyset$ then
$|A\cup B|=|A|+|B|$.
\end{defn}

\begin{prop}
If $\cal A$ is an algebra of a set $\Omega$ and $|\ |$ a measure on $\cal A$
then we have:\\
(i) $A,B\in\cal A$ and $A\subseteq B$ implies that $|B\backslash A|=|B|-|A|$.
\\
(ii) If $A,B_1,\ldots ,B_n\in\cal A$ and $A\in\cup_{i=1}^n B_i$ then 
$|A|\le \sum_{i=1}^n |B_i|$.
\end{prop}
\begin{proof} We express $B$ as the disjoint union $(B\backslash A)\cup A$.
As for (ii), we have pairwise disjoint sets $D_i\in\cal A$ 
with the same union as
the $B_i$ on writing $D_1=B_1$, $D_2=B_2\backslash B_1,\ldots ,
D_n=B_n\backslash B_{n-1}$, thus $|\cup_{i=1}^n B_i|=\sum_{i=1}^n|D_i|
\le \sum_{i=1}^n|B_i|$. Now
\[|A|=|\cup_{i=1}^n (D_i\cap A)|= \sum_{i=1}^n |D_i\cap A|\le 
\sum_{i=1}^n|D_i|.\]
\end{proof}

\begin{prop} If $G$ is any group then the collection
$\cal C$ of finite unions of left cosets and right cosets of finite
index subgroups of $G$ (along with $\emptyset$) is an algebra of $G$.
\end{prop}
\begin{proof}
On taking $A=\cup_{i=1}^m C_i$, where $C_i$ is a coset (possibly left,
possibly right) of the finite index subgroup $H_i$, we can consider
$H=H_1\cap\ldots\cap H_m$ which is of finite index in $G$, and we see that
each $C_i$ is also a finite union of (left or right, as above) cosets of $H$.
Then we drop down to the core $R$ of $H$ in $G$, which is also
of finite index but is normal in $G$ too. Now each $C_i$ is just a finite
union of cosets of $R$, where $gR=Rg$ for all $g\in G$. Thus
$\cup_{i=1}^m C_i$ is just a subset of the finite quotient group
$G/R$. Next if we are given $B=\cup_{i=1}^n D_i$ which is also in 
$\mathcal C$, we have $S$ normal and finite index in $G$ (for which we
write $S\unlhd_f G$) with $\cup_{i=1}^n D_i$ a subset of $G/S$. But 
$R\cap S\unlhd_f G$ and $A,B$ are both subsets of $G/(R\cap S)$, meaning that
$A\cup B$ and $A\backslash B$ are too.
\end{proof}

\begin{prop} For any group $G$ and the algebra $\cal A$ of $G$ as defined
in Proposition 2.4, let $A\in\cal A$ consist of $r$ elements of the
quotient group $G/R$, where $R\unlhd_f G$ has index $i$. Then setting
$|A|=r/i$ gives a measure on $\cal A$ with $|G|=1$.
\end{prop}
\begin{proof}
On taking disjoint $A$ and $B$ with $A$ consisting of $r$ elements of $G/R$
and $[G:R]=i$ but $B$ equal to $s$ elements of $G/S$ with $[G:S]=j$, we
have that $A$ is $rk$ elements of $G/(R\cap S)$ where $k=[R:R\cap S]$ and
$B$ is $sl$ elements for $l=[S:R\cap S]$. But these subsets of $G/(R\cap S)$
must be disjoint, so $A\cup B$ is $rk+sl$ elements, giving $|A\cup B|=
rk\cdot 1/ki+sl\cdot 1/lj=|A|+|B|$.
\end{proof}
Clearly multiplying a measure by a strictly positive constant results in
another measure. We will adopt the convention that $|G|=1$ when
considering arbitrary groups, but when we are restricting ourselves
purely to finite groups we assume $|A|$ is just the number of elements
in $A$ for consistency of notation with finite sets.

\begin{prop} For any group $G$ and the algebra $\cal A$ of $G$ as above,
we have $gA,Ag\in\cal A$ with $|gA|=|Ag|=|A|$ for all $A\in\cal A$ and
$g\in G$, as well as $|A^{-1}|=|A|$.
We also have for $A,B\in\cal A$ with $A\subseteq B$ that
$|A|=|B|$ implies that $A=B$.
\end{prop}
\begin{proof} If $A$ is a subset of $G/R$ as above then the (left)
action of $G$ on $G/R$ by multiplication on the left or 
the (right) action of multiplication on the
right preserves the size of subsets. Similarly as $gR$ and $(gR)^{-1}=
g^{-1}R$ are both cosets of $R$, we have $|A|=|A^{-1}|$.
Moreover if $B$ is a subset of $G/S$
as above then $A$ and $B$ are both subsets of $G/(R\cap S)$ with $A$ in
$B$. Therefore $A\neq B$ means that $|B|\ge |A|+|G|/[G:R\cap S]$.
\end{proof}

\begin{defn} Given a group $G$, we say the {\bf basic measure}
$(\cal B,|\,|)$ on $G$ is the one defined in Proposition 2.5. We
say that a measure $(\cal A,||\,||)$ on $G$ is {\bf acceptable}
if it is left invariant (meaning that $gS\in{\cal A}$ with
$||gS||=||S||$ for all $S\in{\cal A}$), or right invariant,
and if it extends the basic measure, that is $\cal B\subseteq\cal A$.
\end{defn}

Note that if $G=g_1H\cup\ldots \cup g_iH$ for $H$ a subgroup of index $i$
then $||g_1H||=\ldots =||g_iH||$ so that $||\,||$ agrees with
$|\,|$ on $\cal B$ up to a constant factor (if $||\,||$ is not
identically zero). Examples of acceptable measures extending the basic
measure occur whenever $G$ is infinite amenable. Another case is that
of Hartman measure, as explained in \cite{wink}: here $G$ is a
topological group with a continuous homomorphism $\theta$ into a
compact group $C$ (such as $G$ a finitely generated group with the
profinite topology and $C$ its profinite completion). Thus $C$ has
Haar measure $\mu$ and a continuity subset $S$ is a measurable subset
of $C$ such that the topological boundary $\partial S$ has zero
measure. We then have an algebra of subsets $A$ of $G$ consisting of the
pullback $A=\theta^{-1}(S)$ of continuity subsets $S$ of $C$ and on
setting the measure of $A$ to be $\mu(S)$, we see that this is well defined.

The finite residual $R_G$ of an infinite group $G$ is defined to be the
intersection of all finite index subgroups of $G$ and we say that $G$ is
residually finite if $R_G$ is the trivial group $I$. We have that $G/R_G$
is residually finite for any group $G$, so when considering the measure on
$G$ as defined in Proposition 2.5, we may as well assume that $G$ is
residually finite, because the measurable subsets of $G$ are just the pullback
of the measurable subsets of $G/R_G$. 
Now $G/R_G$ may well be infinite, which is equivalent to 
$G$ having infinitely many subgroups of finite index, but if $G/R_G$ is
finite then our measure is nothing other than (normalised) counting measure
on $G/R_G$ and we are firmly in the realm of finite groups. In particular
if $G$ has no proper finite index subgroups then we have a trivial algebra
consisting of $G$ with $|G|=1$ and $\emptyset$ with $|\emptyset|=0$. However
many infinite groups are known to be residually finite, for instance finitely
generated linear groups over any field.

\section {Growth of subsets}

For $A,B$ subsets of a group $G$ with the basic measure $(\cal B,|\,|)$,
we let the product set
$AB=\{ab:a\in A, b\in B\}$. We have that if $A,B\in\cal B$ then so is
$AB$ because we can work in $G/(R\cap S)$ as above. We are interested in
comparing $|AB|$ with $|A|$ and $|B|$. For a finite group 
we certainly have $|AB|\le |A||B|$
but in this section we are especially interested in when $|AB|$
is small. There is much on this subject, going back to Cauchy for abelian
groups right up to recent interest in the non abelian case; see for
instance \cite{taop} and also \cite{taob} Section 2, especially 2.7 for
the non commutative case.
Here we summarise the results we
need, giving proofs when we generalise from finite groups
to the measurable case. The first is a non abelian version of \cite{taob}
Proposition 2.2.
\begin{prop}
For $A$ and $B$ measurable non empty subsets of a group $G$ with the
basic measure, we have
$|AB|=|A|$ if and only if there is a finite index subgroup $H$ of $G$
such that $A$ is a union of left cosets of $H$ and $B$ is contained in
a right coset of $H$. The subgroup $H$ can be taken to be $S_R(A)$, the
stabiliser of $A$ under right multiplication of $G$.
\end{prop}
\begin{proof}
First if $A=x_1H\cup\ldots\cup x_kH$ and $B\subseteq Hy$ then
$AB\subseteq AHy=Ay$. Conversely
for any choice $b_0\in B$ we have $|Ab_0|=|A|=|AB|$ and $Ab_0\subseteq AB$
so that $Ab_0=AB$ by Proposition 2.6. Thus $Bb_0^{-1}$ is in $S_R(A)$,
which has finite index in $G$ because the orbit of $A$ consists of elements
of $G/(R\cap S)$ and so is finite. Consequently $B$ is in the right coset 
$S_R(A)b_0$. But $aS_R(A)\subseteq A$ for
all $a\in A$ so $A$ is a union of left cosets of $S_R(A)$.
\end{proof}
We now look at Ruzsa distance, introduced in \cite{rz} for abelian groups
and generalised in \cite{taop} for non abelian groups.

\begin{defn} Let $A$ and $B$ be non empty measurable sets of a group $G$
with the basic measure.
The (left) Ruzsa distance is defined as
\[\mbox{d}(A,B)=\mbox{log}\ \frac{|AB^{-1}|}{|A|^{1/2}|B|^{1/2}}.\]
\end{defn}
We have $\mbox{d}(gA,gB)=\mbox{d}(A,B)$ and $\mbox{d}(Ag,Bg)=\mbox{d}(A,B)$ 
too. It also satisfies the symmetric and triangle inequality,
as shown in \cite{taop} Lemma 3.2, but fails the zero axiom in both
directions in that we can have d$(A,A)>0$ (in which case 
d$(A,B)\ge\mbox{d}(A,A)/2$ for all non empty measurable $B$)
and we can have d$(A,B)=0$ but $A\neq B$, as will be seen below.
Also we 
do not have d$(A,B)$ equal to d$(A^{-1},B^{-1})$ in general as $|AB^{-1}|$
need not equal $|B^{-1}A|$: indeed we
might not even have $|AA^{-1}|=|A^{-1}A|$. This is well known, for instance
when $A=H\cup xH$ for $H$ a subgroup of $G$ and $xHx^{-1}\neq H$. Here we
quantify this point.\\
\hfill\\
{\bf Example}: If $H$ is a subgroup of the finite group $G$ then
$AA^{-1}=H\cup xH\cup Hx^{-1}\cup xHx^{-1}$ and $A^{-1}A=H\cup HxH\cup
Hx^{-1}H$. Thus $|AA^{-1}|\le 4|H|-1$ but the basic theory of double
cosets tells us that $|HxH|=|H|^2/|H\cap xHx^{-1}|$.
For instance if $G$ is the symmetric group $Sym(2n)$ and $H$ is the subgroup
fixing $n+1,n+2,\ldots ,2n$ then we can take $x=(1,n+1)(2,n+2)\ldots 
(n,2n)$, so that $xHx^{-1}$ fixes $1,2,\ldots ,n$ with $|HxH|=|H|^2$.
Thus we have $AA^{-1}=
4(n!)-2$ but $|A^{-1}A|=(n!)^2+n!$, with $|AA^{-1}|/|A^{-1}A|$ tending to
0 as $n$ tends to infinity. 

However we can get round this
by defining right Ruzsa distance with $|B^{-1}A|=|A^{-1}B|$ 
in place of $|AB^{-1}|$.
This is also left and right invariant under multiplication by group
elements, and if required we can add both to get double Ruzsa distance
dd$(A,B)$ which will still be symmetric and satisfy the triangle inequality.
The next statement for finite groups
is Proposition 2.38 in \cite{taob} (with proof left as
an exercise).
\begin{prop}
We have d$(A,B)=0$ if and only if $A=gH$ and $B=\gamma H$ for $g,\gamma
\in G$ and $H$ some finite index subgroup of $G$.
\end{prop}
\begin{proof}
This implies that $|A|=|B|=|AB^{-1}|$. Therefore by Proposition 3.1 we
have that there is a finite index subgroup $H$ such that $A$ is a union of
left cosets of $H$ and $B^{-1}$ is contained in a right coset of $H$. 
This implies $|A|\ge |H|\ge |B|$ so all are equal and we must have 
$g,\gamma\in G$ such that $A=gH$ and $B^{-1}=H\gamma^{-1}$ so $B=\gamma H$.
\end{proof}
Of course if $A=gH$ is a left coset of $H$ then it is also a right coset,
but of $gHg^{-1}$. We make the following definition:
\begin{defn} We say that $A$ is a {\bf left right coset} in a group $G$
if it is both a
left coset and a right coset of the same subgroup $H$ in $G$.
\end{defn}
%Note that a set
%can only be a right coset of two subgroups if these subgroups are equal
%(by taking the coset containing the identity in each case).
%Thus 
%a subset $A$ is a left right coset of the subgroup $H$
%if and only if we have 
This is equivalent to saying we have
some $g\in G$ for which $A=gH=Hg$, 
which holds if and only if
$g$ is in the normaliser $\mbox{Norm}(H)$.

\begin{co}
We have d$(A,A)=0$ if and only if $A$ is a left coset of a finite index
subgroup. Also d$(A,A^{-1})=0$ if and only if $A$ is a left right coset
of a finite index subgroup.
\end{co}
\begin{proof}
The first statement is immediate from Proposition 3.3
and if $A=gH$ with $A^{-1}=\gamma H$ then $A=H\gamma^{-1}$,
so $A$ is a right coset of $H$ containing $g$.
\end{proof}

We also have similar results for the right Ruzsa distance on swapping left
and right cosets. Moreover we can put these together for the double
distance.
\begin{co}
We have dd$(A,B)=0$ if and only if there exists a finite index subgroup
$H$ and $g,\gamma\in G$ with $\gamma^{-1}g\in\mbox{Norm}(H)$ such that
$A=gH$ and $B=\gamma H$.
\end{co}
\begin{proof} 
We have finite index subgroups $H, L$ of $G$ with 
$A=gH=(gHg^{-1})g=Lg_1$ and
$B=\gamma H=(\gamma H\gamma^{-1})\gamma=L\gamma_1$. 
Now a set
can only be a right coset of two subgroups if these subgroups are equal
(by taking the coset containing the identity in each case), so
we have
$L=gHg^{-1}=\gamma H\gamma^{-1}$, with $gH=Lg$ and $\gamma H=L\gamma$.
\end{proof}
This means that although double Ruzsa distance is not a metric, it is
not so far away from being so. 
For instance it is on measurable subsets 
which are cosets of any finite index subgroup that equals its own
normaliser. Also if we restrict to subgroups, we have 
immediately from Proposition 3.3: 
\begin{co}
Ruzsa distance is a metric on the finite index subgroups of any group $G$.
\end{co}
Here it does not matter if we use left or right Ruzsa distance, which will
be equal, or double Ruzsa distance which is twice each of these.

We now consider growth in groups, which has been much studied for finite
subsets. The idea is to take any measurable set $A$
and consider the sizes $A,A^2,A^3$ and so on.
This occurs in the theory of covering
groups by conjugacy classes, which involves taking a group $G$, usually
a finite simple group, and a non trivial conjugacy class $C$ so as to
determine the smallest value of $n$ such that $C^n=G$. The monograph
\cite{arstherz} discusses the theory in this area, and generalises $C$
to any subset of an arbitrary finite group in Chapter 3, obtaining
results not unlike those that follow.

\begin{thm}
Given a group $G$ with the basic measure $(\cal B,|\,|)$
and non empty $A\in\cal B$, consider the sequence of
measurable sets $A^2,A^3,\ldots $. We have:\\
(i) There exists $n$ such that $|A^n|=|A^{n+1}|$ and this implies that
$|A^{n+1}|=|A^{n+2}|=\ldots$.\\
(ii) We have $|A^n|=|A^{n+1}|$ if and only if $A^n$ is equal to a left right
coset $gH$.\\
\end{thm}
\begin{proof}
(i) If we express $A$ as a subset of $G/R$ then so is
$A^2,A^3,\ldots$ so we will assume that $G$ is a finite group. As
$A\neq\emptyset$, we have that $aA^n\subseteq A^{n+1}$ for any $a\in A$
so $|A^n|$ is a non decreasing sequence with values in
$\{0,1/[G:R],2/[G:R],\ldots ,1\}$ which must eventually
stabilise. But if $|A^n|=|A^{n+1}|$ then, on taking a particular
$\alpha\in A$, we have $|A^n|=|A^n\alpha|=|A^{n+1}|$ with 
$A^n\alpha\subseteq A^{n+1}$ so $A^n\alpha=A^{n+1}$ by Proposition 2.6.
Consequently we have $aA^n\subseteq A^n\alpha$ for all $a\in A$. Thus
for $a_1,a_2\in A$ we have $a_1a_2A^n\subseteq a_1A^n\alpha\subseteq 
A^n\alpha^2$, so $A^{n+2}\subseteq A^n\alpha^2$ and $|A^{n+2}|=|A^n|=
|A^{n+1}|$.

For (ii), first suppose that $|A|=|A^2|$, in which case we have that the
Ruzsa distance d$(A,A^{-1})=0$.
Therefore $A$ is a left right coset by Corollary 3.5. The general
result follows on setting $S=A^n$ because $|A^n|=|A^{n+1}|$ implies that
$|A^n|=|A^{2n}|$ from (i), thus $|S|=|S^2|$ giving $S=gH=Hg$.
Conversely if $A^n$ is a left right coset $gH=Hg$ then $A^{2n}=gH\cdot Hg
=gHg=Hg^2$ so $|A^n|=|A^{2n}|$. As $|A^n|$ is increasing, we have
$|A^n|=|A^{n+1}|$.
\end{proof}

This gives us a complete answer as to which measurable subsets $A$ of
$G$ ``expand to become the whole group''; meaning that there exists $n>0$
such that $A^n=G$ (and then of course $A^m=G$ for all $m\ge n$).
\begin{co}
Given a non empty measurable subset $A\in\cal B$ for $G$ equipped
with the basic measure $(\cal B,|\,|)$,
we have that $|A^n|<|G|$
for all $n\in\N$ if and only if $A$ is a subset of a left right coset of
a proper finite index subgroup $H$ of $G$.
\end{co}
\begin{proof}
If $A\subseteq gH=Hg$ for $H$ a proper finite index subgroup
then $A^n\subseteq g^nH=Hg^n$ which is still a coset
of $H$, thus never equal to $G$. Conversely if $|A^n|=|A^{n+1}|<|G|$, so that
$A^n=gH=Hg$ for $H$ a proper subgroup of $G$, then take any $b\in A^{n-1}$.
For all $a\in A$ we have $ba\in A^n=gH$ so that $A$ is contained in the
coset $b^{-1}gH$ which must be equal to $aH$ for any $a\in A$.
Now $A^n=gH$ so $A^{n+1}=agH=gHa$, meaning that (for $x=g^{-1}ag$)
we have $xH=Ha$, with $a\in xH$ so $aH=Ha$.
%as we are effectively
%working in a finite group we can assume that there is $k\in\N$ with
%$g^k=e$ for all $g\in G$. Now $|A^{kn}|=|A^{kn+1}|=|(A^{kn+1})^2|$ so 
%$A^{kn+1}$ is a left right coset of some subgroup $L$. But $A^{kn}=Hg^k$
%and for $\alpha\in A$ we have $A^{kn+1}=Hg^k\alpha$ as before, so $A^{kn+1}$
%is a coset of $H$ and $L=H$. But $a=a^{kn+1}\in A^{kn+1}$ for all $a\in A$,
%thus $A\subseteq\gamma H=H\gamma$ and so is also contained in a left right
%coset of the same subgroup where stabilisation of $|A|$ occurs.
\end{proof}

One might think that any generating set for $G$ expands to become all
of $G$ but Corollary 3.9 illustrates a dichotomy even
within finite groups in
that some groups $F$ can have a set $S$ which generates $F$ but such that
$S^n$ is a proper subset of $F$ for all $n\in\Z$. Certainly if
$\langle S\rangle=F$ then $\cup_{n\in\Z} S^n=F$ (where we take
$S^0=\{e\}$) and even $\cup_{n\in\N}S^n=F$ because $F$ is finite, so that
$S^{-1}\subseteq S^{|F|-1}$. Moreover if $e\in S$ then $S^m\subseteq S^n$
so that $F=S^N$ for some $N>0$. But this need not be true for arbitrary
generating sets if $F$ is not perfect (meaning that its commutator subgroup
$F'$ is a proper subgroup of $F$). For instance Theorem 4.10 of
\cite{nath} claims that if $G$ is a finite abelian group and $A$ a subset
that generates $G$ then there exists $N>0$ with $A^N=G$. In fact any finite
abelian group has such an $A$ with $A^N\neq G$ for all $N$. The problem
is the first three lines: we cannot assume without loss of generality
that the identity is in $A$.

\begin{thm}
If $F$ is a finite group then there exists a generating set $S$ for $F$ with
$S^N\neq F$ for all $N\in\Z$ if and only if $F$ is not perfect. There
exists a symmetrised generating set $S$ (one where $S^{-1}=S$) 
with $S^N\neq F$ for all $n\in\Z$ if and only if
$F$ surjects to the cyclic group $C_2$.
\end{thm}
\begin{proof}
Let $\langle S\rangle=F$ and consider the sequence $|S|,|S^2|,|S^3|,\ldots$.
Either we have some $n$ with $|S^n|=|F|$ or some $n$ with $|S^n|=|S^{n+1}|
<|F|$ but the latter occurs exactly when $S\subseteq gH=Hg$ for $H$ a proper
subgroup of $F$ by Corollary 3.9. But $S$ is a generating set contained in
$\langle g,H\rangle$ so we must have $\langle g,H\rangle=F$. Now 
$gHg^{-1}=H$ means that $H$ is normal in $F$ with $F/H$ non trivial and
cyclic. Thus if $F$ has no non trivial abelian quotients then such a
generating set $S$ cannot exist, whereas if $F$ is not perfect we have
$H\unlhd F$ with $F/H$ cyclic and non trivial. Thus we can take $S=gH$
where $gH$ generates $F/H$.

Now suppose that $F$ has a subgroup $H$ of index 2 then the non trivial
coset $S=gH=Hg$ generates $F$ and is symmetric, but $S^n$ is either
$H$ or $gH$. Conversely if we have $S=S^{-1}$ with $\langle S\rangle=F$
but $S^n\neq F$ for any $n\in\N$ then again
$S\subseteq gH=Hg$ for $H$ a proper
subgroup of $F$. But now $\langle g,H\rangle=F$ with
$H\unlhd F$ and $S^{-1}\subseteq Hg^{-1}=g^{-1}H$. As $gH\cap g^{-1}H=
\emptyset$
unless $gH=g^{-1}H$, we have $g^2\in H$ with $[F:H]=2$. 
\end{proof}
Here is another description of the sets in Corollary 3.9.
\begin{co}
Given a non empty measurable subset $A$ of $G$, we have that $|A^n|<|G|$
for all $n\in\N$ if and only if $A$ is a subset of a coset of a proper
normal subgroup of $G$, or $A$ is contained in a maximal subgroup
of $G$. If $G$ is perfect then the latter always occurs. 
\end{co}
\begin{proof}
Either $A$ is a generating set or is contained in
a maximal subgroup $L$ of $G$. If the former happens then by the
proof of Theorem 3.10, we have $A\subseteq gH$, where $H$ is a proper normal
subgroup of $G$ and $G/H$ is cyclic.
\end{proof}

\section{Doubling and product free sets}

If $A$ is a finite subset of a group $G$ then there has been much study
on how the size of $A^2$ can vary with that of $A$, 
as we shall see later. Here we will concentrate on those results for finite
groups which provide applications for measurable groups.

In \cite{bherz} a subset $B$ of the finite
group $G$ is called a basis for $G$ if $B^2=G$. The aim is then to find
a basis for a given group which is as small as possible. However we will
look at the two related ideas of finding subsets $A$ of $G$ which are as 
large as possible without being a basis for $G$, and of having a lower
bound for $|A|$ which ensures that it will be a basis. The next
proposition is an old and well known fact.
\begin{prop}
If $G$ is a group of order $n$ and $A$ is a subset of $G$ with 
$|A|>n/2$ then $A^2=G$.
\end{prop}
\begin{proof}
For any $g\in G$ we have $|gA^{-1}|+|A|>n$ so $gA^{-1}$
and $A$ meet, giving $g\in A^2$.
\end{proof}
However we will now see that this value is best possible. For groups of odd
order this is straightforward.
\begin{prop}
If $G$ is a group of order $n=2k+1$ then we have a set $A$ of size $k$
such that $A^2\neq G$.
\end{prop}
\begin{proof}
Pair off each element with its inverse, leaving out the identity $e$. On
taking $A$ to be one element from each of these pairs, we have that
$e\not\in A^2$.
\end{proof}
In fact for some very small groups $G$ it can happen that no set $A$
of size $k=(|G|-1)/2$ or $k=|G|/2$ (for $|G|$ odd or even respectively)
has $A^2=G$.
An argument in \cite{bherz} shows that they always
exist for $k\ge 6$. Smaller values of $k$ can then be quickly tested on
computer, giving that such subsets do not exist precisely 
for the cyclic groups of orders 1 to 7 and abelian groups of orders 
4 and 8.

%If we allow ourselves to invoke the odd order theorem then we further have:
%\begin{thm}
%If $G$ is a group of order $n=2k+1$ then we have a subset $S$ of $G$
%with $|S|=k$ but $|S^2|\le 2k-1$.
%\end{thm}
%\begin{proof}
%The answer is clear if $G$ is cyclic, so we proceed by induction on $k$. As
%$G$ is soluble, we have $N\unlhd G$ with $G/N$ cyclic of odd order $r$
%and $|N|=s$ for $rs=2k+1$. We are done if $N$ is trivial. Otherwise
%choose $(s-1)/2$ cosets of $N$ in $G$, as in Proposition 4.2, so that we
%have a set $T$ of $r(s-1)/2$ elements of $G$ with $N\cap T=\emptyset$.
%By the inductive hypothesis, we can now add $(r-1)/2$ elements of $N$ to
%$T$, giving $S$ such that $S^2$ misses out at least 2 elements of $N$.
%\end{proof}

If rather than considering $|A^2|$ in terms of $|A|$, we consider subsets
$A,B\subseteq G$ and look at how $|AB|$ can depend on $|A|$ and $|B|$ then
there are many known results. First note that Proposition 4.1 generalises
immediately to $AB=G$ if $|A|+|B|>|G|$. Next there is the 
Cauchy-Davenport theorem stating that in the cyclic group $G$ of order $p$
for $p$ prime, we have
$|AB|\ge\mbox{min}(p,|A|+|B|-1)$ for $A$ and $B$ non empty subsets of $G$.
This was generalised by Kneser in
1953 to abelian groups $G$ where he showed that if $|A|+|B|\le |G|$ then
either $|AB|\ge|A|+|B|$ or $|AB|=|A|+|B|-|H|$ for $H$ a proper subgroup
of $G$ and $A$, $B$ are both unions of cosets of $H$. In particular if
$G$ is abelian and of odd order $2k+1$ with $|A|=|B|=k$ but $|AB|<2k$
then $|H|$ must divide $k$ so is equal to 1, giving $|AB|=2k-1$. 
In the case when $G$ is a non abelian group, Kemperman showed in \cite{kp56} 
that for any $a\in A$ 
and $b\in B$, there exists a subgroup $H(a,b)$ of $G$ such that 
$aHb\subseteq AB$ and $|AB|\ge |A|+|B|-|H|$.

If we now look at $|AB|$ when $A$ varies over subsets of a fixed size $r$ 
and $B$ over size $s$ then this is covered in a series of papers by
Eliahou and Kervaire. In particular if $G$ is any finite group and
$1\le r,s\le n=|G|$ then
we can define $\mu_G(r,s)$ to be the minimum of $|AB|$ for $A,B\subseteq G$
with $|A|=r$ and $|B|=s$. In \cite{ekp} with Plagne it was shown that
\begin{equation}
\mu_G(r,s)=\mbox{min}_{d|n}\left(\lceil r/d\rceil+
\lceil s/d\rceil-1\right)d
\end{equation}
when $G$ is abelian. In \cite{ekeu} it is shown that
if $G$ is solvable then for all values of $r$ and $s$ above with $r\le s$
we have $A\subseteq B$ with $|A|=r$ and $|B|=s$ but $|AB|\le |A|+|B|-1$.
In particular if $|G|=2k+1$ and $r=s=k$, we can invoke the Odd Order
Theorem of Feit-Thompson to improve Proposition 4.2 by getting
a subset $A$ of size $k$ with $|A^2|\le 2k-1$.
However for general groups we cannot always take $A=B$ to obtain
$\mu_G(r,s)$, for instance it is 
mentioned here that
$\mu_{Alt(4)}(6,6)=9$ but $|A^2|\ge 10$ for $A$ a 6 element subset of $Alt(4)$.

As for a lower bound on $\mu_G(k,k)$ when $|G|=2k+1$, we have from \cite{ekdm}
that $\mu_G(r,s)$ is at least the right hand side of (1) when $G$ is
solvable, so for $r=s=k$ this gives us that $\mu_G(k,k)=2k-1$.  
To complete the picture in this particular case, we have:
\begin{prop}
If $G$ is a finite group of order $2k+1$ and $A$ and $B$ are subsets of
size $k$ then we have $|AB|=2k-1,2k$ or $2k+1$ and all three values occur
except for the cyclic group $C_3$ where we only obtain $2k-1$, $C_5$ where
we only obtain $2k-1$ for $A=B$ and $2k-1,2k$ otherwise, and
$C_3\times C_3$ where we have $2k-1$ and $2k+1$ for $A=B$ and all three
values otherwise.
\end{prop}
\begin{proof}
By the above results and basic calculations, we are done if
we show that there exists $A\subseteq G$ with $|A|=k$ and $A^2$ having
size exactly $2k$ unless $G=C_3, C_3\times C_3$ or $C_5$.

We proceed by induction on $k$ and use the Odd Order Theorem to say that
$G$ has a normal subgroup $N$ of order $2r+1$, where $|G|=(2r+1)(2s+1)$
with $r,s\ge 1$ (unless $G$ is cyclic, in which case we can take
$A=B=\{1,3,\ldots ,2k-3\}\cup\{0\}$).
First assume that $N$ is not equal to any of our three troublesome cases.
By induction we have $S\subseteq N$ with $|S|=r$ such that $S^2$ misses out
precisely one element $n$ from $N$. 
We then take all the elements from $s$ other cosets 
$x_1N, x_2N,\ldots ,x_sN$ of 
$N$ so that no $x_iN$ and $x_i^{-1}N$ are chosen together, as in Proposition
4.2. Placing these along with $S$ to obtain $A\subseteq G$ with
$|A|=(2r+1)s+r=k$, we have that if $a,b\in A$ with $ab\in N$ then
$a,b\in S$ so $A^2\cap N$ is also all of $N$ except $n$.
However if we have a coset $yN$
with $y\not\in N$ then there are two elements $cN,dN$ from the set
$\{N,x_1N,\ldots ,x_sN\}$ of size $s+1$
with $cNdN=yN$ by Proposition 4.1. Let us define subsets $C,D$ of $A$ where
$C=cN$ and $D=dN$, unless $c$ (or $d$) is the identity $e$ in which case
we let $C$ (or $D$) be equal to $S$. However $c$ and $d$ are not both $e$, 
so $|C|+|D|>|N|$ and hence $CD$ is all of $yN$ (by considering
$d^{-1}c^{-1}Cd$ and $d^{-1}D$ as subsets of $N$).

If however $N$ is troublesome then we can form $A$ as before but then
we find a non trivial coset $xN$ and remove $x$ from $A$, so that
$|A\cap xN|=2r$. We then select just the one element $x^{-1}$ from
its inverse $x^{-1}N$ in $G/N$ and put this back in $A$.
Provided that $r\ge 2$ the argument works 
as before because $|A\cap xN|+|A\cap N|=3r>2r+1$,
and the elements of $N$ in $A^2$ can only come from
$(xN\cap A)x^{-1}$, $x^{-1}(xN\cap A)$ and $(A\cap N)(A\cap N)$. But the
first two sets are equal to $N-\{e\}$ so we can choose $A\cap N$ to consist
of $r$ elements such that $(A\cap N)(A\cap N)$ misses $e$ by Proposition
4.2.

Now suppose that $r=1$, so that $N=C_3$. This argument will still ensure
that $A^2\cap N=N-\{e\}$ but now we have $|A\cap xN|=2$ and
$|A\cap N|=1$, meaning that $A^2$ need not cover all of $xN$. However if
we can find $x$ with $xN$ having order greater than 3 in $G/N$ then we
can ensure that we choose all of $x^2N$ as one of our cosets which are
fully contained in $A$, thus $xN=x^{-1}\cdot x^2N$ is in $A^2$.

Finally if $G/N$ has more than 5 elements then we can find $aN$ and $bN$
disjoint from all of $N,xN$ and $x^{-1}N$ with $aNbN=xN$, so we include
all of $aN$ and $bN$ in $A$ too, leaving only $|N|=|G/N|=3$.
\end{proof} 

Now we move to groups of even order where
it is clear that we can have $H^2=H$ with $|H|=|G|/2$ if
$H$ is a subgroup of index 2 in $G$, or a coset thereof. 
Indeed if a subset $S$ of a finite
group $G$ satisfies $|S^2|=|S|$ then we already know that $S$ must be
a left right coset by Theorem 3.8 (ii). 
However there certainly are many groups of even order 
possessing no subgroup of index 2:
for instance the alternating groups $Alt(n)$ for $n\ge 4$. 
If $r+s=|G|$ then we can easily find $A$ and $B$ of sizes $r$ and $s$ with
$|AB|\ne G$ by setting $B$ equal to the complement of $A^{-1}$, but
we would like to take $A=B$.

The following applies to all groups of even order without needing any
knowledge of finite simple groups.
\begin{thm}
If $G$ has order $2k$ then there exists a subset $S$ of size $k$ with
$S^2\neq G$.
\end{thm}
\begin{proof}
The requirement that there is a $g\in G$ with $g\not\in S^2$ is equivalent
to $S\cap gS^{-1}$ being empty. We will try to find a suitable $g$ and
a symmetric $S$, so that it is enough to confirm $S\cap gS=\emptyset$.

We assume that $g$ is of even order: note that this ensures that $g^i$ is 
never conjugate to $g^j$ if $i$ and $j$ are of different parities.
We further ask that $g^i$ is never a square in $G$ when $i$ is odd. This
can be guaranteed by taking $g$ to have order $m=2^rl$ where $l$ is odd and
$r$ is  as large as possible amongst the elements of $G$.

We now consider the action of left multiplication by $g$ on $G$. This is
a free action so each orbit is a $m$-cycle and there are $2k/m$ of them.
On taking Orb$(x)$ for $x\in G$, the idea is to form $S$ and its complement 
$T=gS$ by placing $g^ix$ in $S$ for $i$ even and $T$ for $i$ odd. However
we need to ensure that inverses are placed together.
If $x^{-1}\in\mbox{Orb}(x)$ then this is fine as $g^i$ cannot equal $x^{-2}$
for $i$ odd. However in general Orb$(x^{-1})$ will be a separate orbit
and hence disjoint from Orb$(x)$. To achieve this, we consider a different
action on $G$ by the group $C_m\times C_2$ where the action of the first 
component is conjugation by $g$ and the second sends $x$ to $x^{-1}$: these
do commute. We then place the orbits under the new action
of $x, g^2x,\ldots ,g^{m-2}x$ in $S$
and those of $gx,g^3x,\ldots ,g^{m-1}x$ in $T$. Note that we are not
claiming these are all distinct orbits or even that the orbits are of the
same size, but this is a well defined procedure unless the orbits of
$g^ix$ and $g^jx$ are equal for $i$ odd and $j$ even. However then we would
have some $k$ such that $g^ix$ is equal to $g^{k+j}x^{\pm 1}g^{-k}$.
The plus sign implies that $x$ conjugates $g^{-k}$ to $g^{i-k-j}$, with
the two indices having different parities, so different orders, whereas
taking the minus sign means that $(g^kx)^2=g^{2k+j-i}$ but $2k+i-j$ is
odd so this is not a square.

Having chosen $x$, we have now partitioned all elements of the form
$g^ix^{\pm 1}g^j$ for arbitrary $i$ and $j$, so we now choose any $y$
not in this set and start again. We now need to show that for any
$\gamma \in G$ we have that $\gamma^{-1}$ is in the same subset and
$g\gamma$ is not. This ensures that $S\cap T=\emptyset$ with $T=gS$, and
as $S\cup T=G$ we obtain $|S|=k$. We express $\gamma$ as $g^iz^{\pm 1}g^j$
where $z\in G$ was one of the elements picked out during the above process.
We see that $\gamma$ is in the orbit of $g^{i+j}z$ or $g^{-i-j}z$, so
certainly $\gamma^{-1}$ is in the same orbit and $g\gamma$ will be in the
orbit of $g^{i+j+1}z$ or $g^{-i-j-1}z$ respectively, either of which end up
on the other side.
\end{proof}
\begin{co}
There exist finite simple groups $G$ and subsets $S$ of $G$ with $|S|=|G|/2$
such that $|G|-|S^2|$ can be arbitrarily large.
\end{co}
\begin{proof} The above construction showing that $g\notin S^2$ also shows
that $g^i\notin S^2$ for odd $i$, so we miss out on at least
$m/2$ elements where
$m$ is the order of $g$. Thus we merely require finite simple groups with
elements of order an arbitrarily high power of 2, and even the alternating
groups provide this.
\end{proof}

Other applications of these results enable us to give counterexamples to 
Freiman type statements with constant 2. The original theorem, due to
Freiman himself in \cite{mark}, is that a subset $A$ of 
$G$ satisfies $|A^2|<(3/2)|A|$ if and only if
$A$ is contained in a left right coset $gH=Hg$ (where we can of course
take $g$ to be any element of $A$, so that $A$ is in the normaliser of
$H$) with $|A|>2/3|H|$ and such that $A^2$ is exactly equal to
$g^2H=Hg^2$. A neat proof is also given in the Tao blog of
November 2009. In \cite{astqu} it is asked how one can characterise
sets $A$ with $|A^2|<2|A|$, but it is clear that the conclusion of
Freiman's theorem can no longer apply: if $A=\{e,g\}$ with
$A^2=\{e,g,g^2\}$ and $g$ has order at least 4 then $A^2$ is not a coset,
because it contains $e$ and so would have to be a subgroup. In fact even
if all elements of $G$ have order 1, 2 or 3 the conclusion is still false:
\begin{prop}
If $G$ is any finite group, except for a cyclic group of order at most 3
or $C_2\times C_2$,
then there exists $A\subseteq G$ with $|A^2|\le (7/4)|A|$ but such that
$A^2$ is not a coset of any subgroup of $G$.
\end{prop}
\begin{proof}
We can assume that the order of any non identity element in $G$ is 2 or 3.
If $G$ has elements of order 3 then let $H_3$ be a Sylow 3-subgroup of
$G$. We certainly have $|A|=4$ but $|A^2|=7$ inside $C_3\times C_3$ and
$H_3$ surjects to this unless $H_3$ is cyclic, in which case it can only
have order 3. Otherwise we pull back $A$ from $C_3\times C_3$ to $H_3$,
so that $A$ is a subgroup of $G$ consisting of 4 cosets of $H_3$ with
$A^2$ equal to 7 cosets out of 9. Then $|A^2|/|G|$ is 7/9 times a power
of 2, so $A^2$ is not a coset in $G$.

%Alternatively if $G$ has odd order $2k+1$ and we take a subset $A$ with
%$|A|=k$ but $|A^2|=2k-1$ then $G$ will not have 
%a subgroup of order $|A^2|$ unless $k=1$. 
%In general when $G$ has even order, the construction 
%in Theorem 4.4 will also work unless the subset $S$ is a subgroup 
%with index 2 or a coset of this. Let us suppose that
%there exists a non trivial $y\in G$ of odd order. We replace $g$ of
%order $m$ in the proof with $g^l$ of order $2^r$ and note that the
%conditions still hold for this new element. On first taking $x=e$,
%we have even powers of $g^l$, including $e$ of order 1,
%in $S$ and odd powers in $T$. However if we next choose $x=g^{-l}y$
%then $y$ now gets placed in $T$, so both $S$ and $T$ contain elements of
%odd order and neither is a subgroup of index 2.

Similarly if a Sylow 2-subgroup $H_2$ of $G$ is non trivial then it
must be $C_2$, $C_2\times C_2$ or it contains $C_2\times C_2\times C_2$
which has a subset $A$ of size 4 such that $|A^2|=7$.
This leaves us with $H_3=C_3$ and $H_2=C_2$ or $C_2\times C_2$, so we are just
left with $S_3$ and $A_4$ where we can have in each case $|A|=|G|/2$ with
$|A^2|=(5/3)|A|$.
\end{proof}

We note that $7/4$ is best possible here even if we exclude finitely many
groups as it was shown in \cite{hpeu}
that for $A\subseteq (C_2)^n$ with $|A^2|<2|A|$ then either $A^2$ is
a subgroup or $|A^2|\ge 7/4|A|$.

Consequently work on extending Freiman's (3/2) Theorem has concentrated
on taking $3/2\le K<2$ and trying to show that there exists a constant
$C(K)$ such that if $A$ is a subset of any finite group $G$ with
$|A|\le K|A^2|$ then there exists some subgroup $H$ of $G$ such that
$A^2$ is the union of no more than $C(K)$ right (say) cosets of $H$.
This cannot hold for all finite groups $G$ when $K=2$ because we can
take long arithmetic progressions; namely $A=\{e,g,\ldots ,g^{n-1}\}$
for $g$ an element of some group with the order of $g$ being at least
$2n$. It clearly does hold if we restrict to only finitely many groups
and, using a major group theory result, we can show here that this is
necessary.
\begin{prop} If $\cal G$ is any infinite family of finite groups then for
any $C>0$ there exists a group $G$ in $\cal G$ and a subset $A$ of $G$
such that $|A^2|<2|A|$ but $A^2$ is not the union of at most $C$ 
right cosets of some subgroup of $G$.
\end{prop}
\begin{proof}
We are done by long arithmetic progressions unless there is an upper
bound on the exponents of all groups in $\cal G$. Moreover if there
are infinitely many groups in $\cal G$ of odd order then we are also
done by Proposition 4.3, because if $|G|=2k+1$ we have $A$ with $|A|=k$
but $|A^2|=2k$ or $2k-1$, so that if $A^2$ is a union of right cosets
of $H\leq G$ then $|H|$ divides $|G|$ and $|A^2|$. But here
we can let $k$ tend to infinity.

Furthermore we are done if the odd part of the orders of the groups $G$ in
$\cal G$ is unbounded, as we can take Sylow subgroups $S$ of $G$. On applying
the above to $A$ in $S$, if $A^2\subseteq S$ is a union of right cosets
of a subgroup $H$ of $G$ then these right cosets are in $S$, so $H$ is a
subgroup of $S$ and the previous argument applies.

We can now restrict $\cal G$ to an infinite sequence $(G_n)$ of
2-groups with bounded exponent. By Zelmanov's solution of the restricted
Burnside problem, the minimum number of generators $d(n)$ of $G_n$ must
tend to infinity with $n$, so $G_n$ surjects to $(C_2)^{d(n)}$ using the
Frattini subgroup. Now in $(C_2)^d$ we have a subset $B$ of size $2^{d-1}$
with $B^2=2^d-1$; for instance take $B$ to be the set with final
coefficient 0 but replacing the element $(1,1,\ldots ,1,0)$ with
$(0,\ldots ,0,1)$. We then pull $B$ back to a union $A$ of cosets in $G_n$
of the kernel $K$ of this surjection. We obtain $|A^2|=(2-(1/2^{d-1}))|A|$
and $|A^2|=(2^d-1)|K|$. Thus if $A^2$ is a union of cosets of some
$H\leq G_n$ then suppose $|G_n|=2^{f(n)}$. We have $|K|=2^{f(n)-d(n)}$
and $|H|$ divides $|G_n|$ and $|A^2|$, thus $|H|$ divides
$2^{f(n)-d(n)}$ meaning that $A^2$ would be at least $2^d-1$ cosets of
$H$.
\end{proof} 

A recent result on product growth was established by Babai,
Nikolov and Pyber in \cite{soda}, following work of Gowers.
This states that for a finite group
$G$ of order $n$, let $m$ be the minimum degree of a non trivial (real)
representation of $G$ into $GL(n,\mathbb R)$. Then for $A,B$ non empty
subsets of $G$ with $|A|=r$ and $|B|=s$
we have $|AB|>\frac{n}{1+(n^2/mrs)}$, with the right hand side at least
the minimum of $n/2$ and $mrs/(2n)$. 
If $r\ge s$ we already have the bound that $|AB|\ge r$, but
$\frac{n}{1+(n^2/mrs)}\ge r$ implies that $0\ge (n-mr/2)^2+r^2(m-m^2/4)$.
Thus we obtain nothing new from this inequality if $m\le 4$, 
which is the case for all finite soluble groups for instance.
However this inequality really comes into
play as $m$ increases. In particular for finite simple groups $G$
we have that $m$ tends to infinity with the order of $G$,
so this shows that although
$|G|-|S^2|$ in Corollary 4.5 can be arbitrarily high, we must have
$|G|/|S^2|$ tending to 1.

This result also solves a question in \cite{ekeu}, which
asks at the start of Section 3 whether the small sumsets property
holds for all finite groups $G$, that is whether for all $1\le r,s\le n=|G|$
we have subsets $A,B$ of $G$ with $|A|=r$, $|B|=s$ and $|AB|\le r+s-1$.
We can find plenty of counterexamples by taking $\lambda$ and $\mu$
such that $|A|=\lambda n,|B|=\mu n$
and assuming that $\lambda+\mu\le 1/2$.
Then we would have $|A|+|B|=(\lambda+\mu)n$ but $|AB|$ is greater than the
minimum of $n/2$ and $mn\lambda\mu/2$, so that if 
$m\ge 2(\lambda+\mu)/
(\lambda\mu)$ then $|AB|>|A|+|B|$. Thus we obtain lots of pairs $(r,s)$
where the small sumsets property fails by 
taking $\lambda$ and $\mu$ as
above, finding a group $G$ such that 
%$\lambda\mu/2(\lambda+\mu)\ge 1/m\ge \lambda\mu$ (for instance
%$\lambda=\mu=4/m$)
such that $m\ge 2(\lambda+\mu)/(\lambda\mu)$ 
and setting $(r,s)=(\lambda |G|,\mu |G|)$.  
 
We now consider these results in the context of infinite measurable
groups $G$ with $|G|=1$.
For such a $G$, let $u$ be the supremum of $|A|$ over
measurable sets $A$ with $A^2$ also measurable and such that $|A^2|<1$. 
From Proposition 4.2 and Theorem 4.4
we immediately have:
\begin{co}
For $G$ any group equipped with an acceptable measure and with $G/R_G$
infinite, we have $u=1/2$ and this is attained if there exists a finite 
quotient of $G$ with even order.
\end{co}
We remark that there do exist interesting examples of infinite
residually finite groups where every finite quotient has odd order.
For instance infinite residually finite $p$-groups (for $p\neq 2$)
have this property. In \cite{ls} Section 5.4, a finitely generated
group is said to be of prosoluble type if the group is
residually finite and all finite quotients are soluble. By the Odd
Order Theorem any example of a finitely generated
residually finite group where all
finite quotients have odd order is of prosoluble type. However such a 
group cannot be linear in characteristic 0 if infinite, by \cite{ls}  
Window 9 Corollary 17 so we can say that any infinite
finitely generated linear group in characteristic 0
has a subset which is, by any 
reasonable definition, half of the group
but whose square is not the whole group, indeed the complement of the
squared set is not negligible. We also remark that all the non trivial
finite quotients of the groups in Theorem 4.11 are
of even order, but none are a power of 2.

Similarly we can look at product free subsets. These are subsets $S$
of a group $G$ such that $S\cap S^2$ is empty. 
This definition makes sense for any subset in any group but a well
studied problem in finite groups is to examine how big a product free
subset can be. Therefore we can ask the same question of infinite
groups with an acceptable measure,
where we restrict $S$ to being a non empty
measurable subset.

The survey article \cite{ked1} is a very readable introduction to
the subject in the case of finite groups
and the author provides an update in \cite{ked2}. We
begin by noting that the obvious upper bound of half the group
in the finite case for a product free subset extends immediately
to the infinite measurable case.
\begin{prop}
If $G$ is a group with an acceptable measure
and $S$ is a measurable product free subset then
$|S|\le 1/2$. If $G$ has the basic measure then we obtain
equality if and only if $S=gH$ for $H$ a subgroup
of index 2 and $g\not\in H$.
\end{prop}
\begin{proof}
For any $s\in S$ we have that $S^2$ contains the measurable set $sS$
with $|sS|=|S|$, so if $|S|>1/2$ then $S$ and $sS$ must meet.
Now suppose $G$ has the basic measure and $S$ is product free with $|S|=1/2$, 
thus forcing $|S^2|=1/2$.
We can then apply Theorem 3.8 (ii).
\end{proof}
Another useful point is Observation 1 in \cite{ked1}, which is that
if $N$ is a normal subgroup of $G$ then we can pull back a product free
subset of $G/N$ to one for $G$. This also works for infinite measurable
groups and the measures of the product free subsets will be the same.
Therefore given any infinite group $G$ with an acceptable measure,
we can get a product free subset of
$G$ by looking at its finite quotients, so it is enough to concentrate
on finite abelian and finite simple groups.

For any measurable group $G$ with $|G|=1$
we let $\alpha(G)$ be the supremum of
$|S|$ over all product free subsets $S$ for $G$. It is shown in 
\cite{alkl} that $\alpha(G)\ge 2/7$ for all finite abelian
groups. Therefore we immediately have $\alpha(G)\ge 2/7$ for all
infinite finitely generated groups which are not perfect, because there
will be a non trivial finite abelian quotient. However in the case of finite 
simple groups it was recently shown by Gowers in \cite{gw} that $\alpha(G)$
can be arbitrarily small. Indeed using the above result of Babai,
Nikolov and Pyber, which was influenced by the Gowers paper, we have that
if $A=B$ with size $r$ then $|A^2|>\frac{n}{1+(n^2/mr^2)}$, so if the
right hand side is at least $n-r$ then $A^2$ and $A$ would have to intersect,
meaning that no product free set can be of size $r$. The latter condition
is equivalent to $\lambda^3/(1-\lambda)\ge 1/m$ where $\lambda =r/n$ so
we see that as $m$ tends to infinity we have $\alpha(G)\rightarrow 0$.
In fact this gives a complete answer to the smallest
size of product free sets
for any infinite family of finite groups.
\begin{prop}
Suppose that $G_n$ is a sequence of non trivial
finite groups then we have that
the infimum of $\alpha(G_n)$ is non zero if and only the minimum degree
$m(G_n)$ of a non trivial real representation is bounded.
\end{prop}
\begin{proof}
If $m(G_n)$ is unbounded then we can apply the result of Babai,
Nikolov and Pyber. However if $m(G_n)$ is bounded above by $N$ then
every $G_n$ has a non trivial image $Q_n$ in $GL(N,\mathbb R)$. By Jordan's
Theorem there exists $j$ depending only on $N$ 
such that $Q_n$ has an abelian subgroup $H_n$ of index at most $j$. Either
$H_n$ is trivial so that $1<|Q_n|\le j$, or we have a product free subset
$S_n$ of $H_n$, so also of $Q_n$ with $|S_n|/|Q_n|\ge 2/(7j)$. In both
cases we pull the product free subset of $Q_n$ back to $G_n$.   
\end{proof}

We can answer the equivalent question for infinite measurable groups by
using the construction in \cite{kn}.
\begin{thm}
There exist infinite finitely generated residually finite groups
equipped with the basic measure such that $\alpha(G)$ is arbitrarily small.
\end{thm}
\begin{proof}
It is shown in Theorem 4 of \cite{kn}
that if $\Gamma=\Pi_{n=1}^\infty S_n$ is
an unrestricted direct product of finite simple groups $S_n$ with 
$l(S_n)\rightarrow\infty$, where $l(S_n)$ is the largest integer $l$ such that
$S_n$ contains the alternating group $Alt(l)$, then $\Gamma$ is the profinite
completion $\hat{G}$ of a finitely generated residually finite group
$G$. We suppose that all $S_n$ are distinct and let $\Delta$ be a finite
index normal open subgroup of $\Gamma$. Then for each $n$ we must have
that $\Delta\cap S_n$ is equal to $S_n$ or $I$, but $\Delta\cap S_n=I$
means that $S_n$ is a subgroup of the finite group $\Gamma/\Delta$, thus
this can only happen for finitely many $n$. As $S_n$ is in $\Delta$ for
all $n>N$ say and $\Delta$ is closed, we have that there is a
surjective homomorphism $\theta$ from $S_1\times\ldots\times S_N$ to
$\Gamma/\Delta$. However this surjective image of a finite direct
product of distinct finite simple groups can only be
$S_{i_1}\times\ldots\times S_{i_k}$ for $1\le i_1<\ldots <i_k\le N$
where $\mbox{ker}(\theta)\cap S_{i_j}=I$ and $\mbox{ker}(\theta)\cap S_m=S_m$
otherwise. This is because $\theta(S_m)$ is normal in in 
$\theta(S_1\times\ldots\times S_N)$, so for $m\neq n$ we have
$\theta(S_m)\cap\theta(S_n)$ is normal in $\theta(S_m)$ and $\theta(S_n)$,
implying that $\theta(S_m)\cap\theta(S_n)=I$ and 
$\theta(S_1\times\ldots\times S_N)\cong\theta(S_1)
\times\ldots\times\theta(S_N)$.

Now suppose that $G$ is the finitely generated residually finite group
with $\hat{G}=\Gamma$ as above. As any (continuous) homomorphism from
$G$ onto a finite group $F$ extends continuously to $\Gamma$, we have
that $F$ must be of the form $S_{i_1}\times\ldots\times S_{i_k}$ too. We
now take any integer $K>0$ and apply Theorem 4 of \cite{kn} with
$\Gamma_K=\Pi_{n=K}^\infty Alt(n)$ where certainly $l(Alt(n))\rightarrow\infty$
as $n$ does. Then the only finite quotients $F$ of the finitely generated
residually finite group $G_K$ are direct products of alternating groups
of rank at least $K$, 
but we have $m(Alt(n))\ge n-1$ for $n\ge 7$
so $\alpha^3(Alt(n))\le 1/m(Alt(n))\le 1/(K-1)$ for
$n\ge K$. This implies that $\alpha(F)\le 1/(K-1)^{1/3}$ too, because
if $F$ had a non trivial representation of dimension less than $K-1$
then this would be non trivial on one of the direct factors.
Consequently $\alpha(G_K)\le 1/(K-1)^{1/3}$ and this tends to 0 as
$K$ tends to infinity.  
\end{proof}

\section{Ruzsa distance on finite index subgroups}

Given any group $G$, we have from Corollary 3.7 that Ruzsa distance
is a genuine metric on the set ${\cal S}(G)$ of finite index subgroups of $G$. 
Here
we will use the double distance formula, so that for subgroups $A$ and
$B$ of $G$ we have
\[\mbox{d}(A,B)=\mbox{log}\frac{|AB^{-1}|}{|A|^{1/2}|B|^{1/2}}
+\mbox{log}\frac{|BA^{-1}|}{|A|^{1/2}|B|^{1/2}}
=\mbox{log}\frac{[G:A][G:B]}{[G:AB]^2}.\]
In order to avoid countless mentions of logs, we
will consider the multiplicative version of Ruzsa distance; that is
\[\mbox{e}(A,B)=\mbox{exp}^{\mbox{d}(A,B)}=\frac{[G:A][G:B]}{[G:AB]^2}\]
with e satisfying $\mbox{e}(A,B)=\mbox{e}(B,A)$, $\mbox{e}(A,B)\ge 1$
with equality if and only if $A=B$, and $\mbox{e}(A,B)\le\mbox{e}(A,C)
\mbox{e}(C,B)$ for any $C\le_f G$.

Now $AB$ need not be a subgroup but we can define 
the ``index'' $[G:AB]$ of $AB$ in $G$;
if $AB$ consists of $k$ cosets of $A\cap B$ in $G$ then
$[G:AB]=k/[G:A\cap B]$. This need not be an integer (consider 2-cycles
in the symmetric group $S_3$ for instance), so that it is not clear
from this formula that $\mbox{e}(A,B)\in\mathbb N$ but we have that
$[A:A\cap B]=[AB:B]$, because if $AB$ is a union of $j$ disjoint cosets
$a_1B,\ldots ,a_jB$ then $a_1(A\cap B),\ldots ,a_j(A\cap B)$ are
disjoint cosets with union $A$. Thus we also have
\[\mbox{e}(A,B)=\frac{[G:A\cap B]^2}{[G:A][G:B]}=[A:A\cap B][B:A\cap B].\]
\begin{prop}
The metric d induces the discrete topology on ${\cal S}(G)$.
\end{prop}
\begin{proof}
We have that $\mbox{d}(A,B)$ takes on values in $\mbox{log}(\mathbb N)$ so
any two distinct points are at least $\mbox{log}\,2$ apart.
\end{proof}
We also need to see when the triangle inequality becomes equality.
\begin{prop} If $A,B,C\le_f G$ then we have $\mbox{e}(A,C)
\mbox{e}(C,B)=\mbox{e}(A,B)$ if and only if $C$ contains the subgroup 
$A\cap B$ and $C$ is equal to $(C\cap A)(C\cap B)$.
\end{prop}
\begin{proof}

We would have  
\[[A:A\cap C][C:A\cap C][B:B\cap C][C:B\cap C]=[A:A\cap B][B:A\cap B].\]
Now let $X=A\cap B\cap C$ and let us multiply top and bottom of the left
hand side by $[Y\cap Z:X]$ for the various intersections $Y\cap Z$ 
of subgroups that appear on this side,
and the same for the right. This gives us
\[[C:X]^2/([A\cap C:X]^2[B\cap C:X]^2)=1/[A\cap B:X]^2\]
so we can take square roots to get $[C:X][A\cap B:X]=[C\cap A:X][C\cap B:X]$.
The right hand side is equal to $[(C\cap A)(C\cap B):X]$ and the left hand side
gives us $[C(A\cap B):X]$. But $C$ is always contained in $C(A\cap B)$ and
itself contains $(C\cap A)(C\cap B)$ so all three must be equal.
\end{proof}
Note that if $A\le C\le B$ for $A$ (and $B,C$) finite index subgroups of
$G$ then $\mbox{d}(A,C)+\mbox{d}(C,B)=\mbox{d}(A,B)$.

We can turn ${\cal S}(G)$ into a graph with a natural metric by joining
the finite index subgroups $A$ to $B$ with an edge if $B$ is a maximal 
subgroup of $A$. We then
regard the edge as an isometric copy of the interval $[0,\mbox{d}(A,B)]$.
This gives us
\begin{prop}
For any finite index subgroups $A$ and $B$ there exists a straight line
in ${\cal S}(G)$ between $A$ and $B$ - that is an isometry
$f:[0,\mbox{d}(A,B)]\rightarrow {\cal S}(G)$ with $f(0)=A$ and
$f(\mbox{d}(A,B))=B$.
\end{prop}
\begin{proof}
We have subgroups $H_0,H_1,\ldots ,H_n$ such that
$A\cap B=H_0<H_1<\ldots <H_n=A$ and we can take this to be a maximal
chain because each subgroup has finite index in the one above. This means
we have edges joining $H_0$ to $H_1$,$\ldots$,$H_{n-1}$ to $H_n$ and
as $\mbox{d}(A\cap B,A)=\mbox{d}(H_0,H_1)+\ldots +\mbox{d}(H_{n-1},H_n)$
by Proposition 5.2,
we have that the join of these edges has length $\mbox{d}(A,A\cap B)$.
We can do the same with $B$ in place of $A$ and then put these two
joins together, with $\mbox{d}(A,A\cap B)+\mbox{d}(A\cap B,B)$ also equal
to $\mbox{d}(A,B)$ by Proposition 5.2.
\end{proof}

However the mere existence of a graph possessing a metric and with vertices
in bijective correspondence with the finite index subgroups of a group
$G$ is not special unless we can say something about how $G$ acts. We
would like an isometric action; even better would be a faithful
isometric action. The first is easy to ensure.
\begin{prop}
The group $G$ acts by conjugation on the graph ${\cal S}(G)$ and this action
is isometric.
\end{prop}
\begin{proof}
We have that if $H\le_f G$ then $gHg^{-1}$ is a subgroup of the same
index, and if $L$ is a maximal subgroup of $H$ then $gLg^{-1}$ is maximal
in $gHg^{-1}$ so they are still joined by an edge of the same length.
\end{proof}

However there are two cases where it is clear that we will have elements
acting trivially: the first is 
that any element in the finite residual
$R_G$ would also be in $N(H)$ for any $H\le_f G$, and hence in 
the kernel $K(G)$ of our action.
Therefore we will say without loss of generality that $G$ is residually
finite by replacing it with $G/R_G$, whereupon
the metric space ${\cal S}(G/R_G)$ is equal to ${\cal S}(G)$. The second 
case is when $G$ has a non trivial centre $Z(G)$ which will
act trivially, so the whole action is trivial if $G$ is an abelian
group. It would be good if $K(G)$, which is the
intersection of the normalisers $N(H)$ over all finite index subgroups
$H$, were equal to the centre when $G$ is residually finite,
but a moment's reflection on finite groups reminds us that every subgroup
of the (non abelian) Quaternion group of order 8 is normal, so here too
is a trivial action. 

In fact this hope is not too far off the truth. For although we have
not seen the group $K(G)$ in the literature, there is a very similar
concept which reduces to $K(G)$ when $G$ is finite. The Baer norm
$B(G)$ is defined to be the intersection $N(H)$ of normalisers over
all subgroups $H$ of $G$. The crucial fact is in \cite{schenk} from
1960: that for any group $G$ the Baer norm $B(G)$ is in the
second centre $Z_2(G)$, which is
formed by quotienting $G$ by its centre and taking the pullback of the 
centre of this quotient $G/Z(G)$ to get $Z_2(G)$.

This has an immediate application to our situation:
\begin{prop}
Let $G$ be a group which is residually finite.
Then the kernel $K(G)$ of our action is contained in $Z_2(G)$.
In particular the action is faithful if and only if
$G$ is residually finite with trivial centre.
\end{prop}
\begin{proof} Given $g$ in $G$, if $g\notin Z_2(G)$  then there
exists $x\in G$ such that the commutator $[g,x]=gxg^{-1}x^{-1}$ is
non trivial in $G/Z(G)$.
Thus there further exists some $y\in G$ such that
the element $z=[[g,x],y]\ne e$ in $G$. Now take a finite quotient 
$Q$ of $G$ with (the image of) $z$ non trivial in $Q$. If $g$ were in
the second centre $Z_2(Q)$ then $[g,x]Q$ is in the centre of $Q$, a
contradiction. Thus $gQ\notin B(Q)$, giving a subgroup 
$S$ of $Q$ with $gSg^{-1}\neq S$. Now pull back to obtain a finite index
subgroup $H$ of $G$ with $gHg^{-1}\neq H$.
\end{proof}

The last part of Proposition 5.5 is relying on the fact that if $g\in K(G)$
and $Q$ is any quotient of $G$ then (the image of) $g$ is in $K(Q)$ too
by the correspondence theorem. Thus for residually finite groups $G$
we have $B(G)$ is trivial if and only if $K(G)$ is trivial too, which
occurs if and only if $G$ has trivial centre. 

\begin{co}
If $G$ is a finitely generated group
then we have a faithful isometric action of $G/K(G)$ on the
connected graph ${\cal S}(G)$ such that every point has a finite orbit.
Moreover $G/K(G)$ is residually finite.
\end{co}
\begin{proof}
The first part follows from the results of this section, along with the
well known fact that a finitely generated group has only finitely
many subgroups of a given finite index. As $Q=G/K(G)$  
acts faithfully on ${\cal S}(G)$, if $q$ is not the identity in $Q$ then
we have $s\in {\cal S}(G)$ with $q(s)\neq s$. But as the orbit of $s$
is finite, the stabiliser of $s$ is a finite index subgroup of $Q$
missing $q$.
\end{proof}

\end{document}